\documentclass[12pt, a4paper]{amsart}
\usepackage[utf8]{inputenc}
\usepackage{amsfonts,amssymb,amsxtra,amsthm,amsmath,amscd,mathrsfs,cite,bbold}
\usepackage[all]{xy}

\usepackage{tikz}

\usepackage[normalem]{ulem}
\usepackage{soul}
\usepackage{color}

\setstcolor{red}

\makeatletter
\@namedef{subjclassname@2010}{%
  \textup{2010} Mathematics Subject Classification}
\makeatother

\def\leq{\leqslant}
\def\geq{\geqslant}
\def\le{\leqslant}
\def\ge{\geqslant}

\theoremstyle{plain}
\newtheorem{theorem}{Theorem}[section]
\newtheorem{proposition}{Proposition}[section]
\newtheorem{lemma}[proposition]{Lemma}
\newtheorem{corollary}[theorem]{Corollary}

\theoremstyle{remark}

\numberwithin{equation}{section}

\addtocounter{footnote}{1}

\begin{document}
\vglue -5mm

\title[ON A SUM INVOLVING general ARITHMETIC FUNCTIONS]
{ON A SUM INVOLVING general ARITHMETIC FUNCTIONS AND THE INTEGRAL PART FUNCTION}
\author{Jing Ma \& Ronghui Wu}

\address{%
Jing Ma
\\
School of Mathematics
\\
Jilin University
\\
Changchun 130012
\\
P. R. China
}
\email{jma@jlu.edu.cn}

\address{
Ronghui Wu
\\
School of Mathematics
\\
Jilin University
\\
Changchun 130012
\\
P. R. China
}
\email{rhwu21@mails.jlu.edu.cn}

\date{\today}

\begin{abstract}Let $f$ be an arithmetic function satisfying some simple conditions. The aim of this paper is to establish an asymptotical
formula for the quantity
\[
S_f(x):=\sum_{n\leq x}\frac{f([x/n])}{[x/n]}
\]
as $x\rightarrow\infty$, where $[t]$ is the integral part of the real number $t$.
This generalizes some recent results of Bordell\`es, Dai, Heyman, Pan and  Shparlinski.
\end{abstract}

	
\maketitle
	
\section{{Introduction}}

As usual, we denote by $\varphi(n)$ the Euler totient function, by $[t]$ the integral part of the real number $t$, by $\log_2$ the iterated logarithm and by $\gamma$ the Euler constant, respectively.
Motivated by the following well-known results
\begin{equation}\label{[x/n]}
\sum_{n\leq x}\left[\frac xn\right]=x\log x+(2\gamma-1)x+O(x^{1/3}),
\end{equation}
\begin{equation}\label{2}
\sum_{n\leq x}\varphi(n)=\frac 3{\pi^2}x^2+O(x(\log x)^{2/3}(\log_{2}x)^{4/3})
\end{equation}
as $x\rightarrow\infty$,
using Bourgain's new exponent pair \cite{J.Bour2017},
Bordell\`es Dai, Heyman, Pan and Shparlinski \cite{BDHPS2019}  proposed to investigate the asymptotical behaviour of the summation function
\[
 \sum_{n\leqslant x}\varphi\left(\left[\frac{x}{n}\right]\right)
\]
as $x\rightarrow\infty$. They proved
\begin{equation}\label{S}
\left(\frac {2629}{4009}\cdot\frac 6{\pi^2}+o(1)\right)x\log x
\leq 
 \sum_{n\leqslant x}\varphi\left(\left[\frac{x}{n}\right]\right)
\leq\left(\frac {2629}{4009}\cdot\frac 6{\pi^2}+\frac{1380}{4009}+o(1)\right)x\log x
\end{equation}
and conjectured that
\begin{equation}\label{S'}
 \sum_{n\leqslant x}\varphi\left(\left[\frac{x}{n}\right]\right)
 \sim\frac 6{\pi^2}x\log x
\end{equation}
as $x\rightarrow\infty$.

The bounds in \eqref{S} have been sharpened by Wu \cite{Wu2019} using the van der Corput inequality \cite{GrahamKolesnik1991}.
Zhai \cite{Zhai2020} resolved the conjecture \eqref{S'} by combining the Vinogradov method with an idea of Goswami \cite{A}.
He established the following asymptotic formula
\begin{equation}\label{zhai}
 \sum_{n\leqslant x}\varphi\left(\left[\frac{x}{n}\right]\right)
 =\frac 6{\pi^2}x\log x+O(x(\log x)^{2/3}(\log_{2}x)^{1/3})
\end{equation}
and showed that the error term in \eqref{zhai} is $\Omega(x)$.
For part of further related works see \cite{LWY-1,LWY-2+,MaSun,MaWu,MaWuzhao,Wu2020note,Zhai2020,ZW, O.Bordelles2020}.

Furthermore, for the Euler totient function $\varphi(n)$, 
Bordell\`es-Dai-Heyman-Pan-Shparlinski \cite[Corollary 2.4]{BDHPS2019} also derived that
$$
S_\varphi(x):=\sum_{n\leqslant x}\frac{\varphi([x/n])}{[x/n]}=x\sum_{m\ge 1}\frac{\varphi(m)}{m^{2}(m+1)}+O(x^{1/2}).
$$
Subsequently, Wu \cite{Wu2020note} sharpened the error term of \cite{BDHPS2019} and proved
\begin{equation}\label{1/3}
S_\varphi(x)=x\sum_{m\ge 1}\frac{\varphi(m)}{m^{2}(m+1)}+O(x^{1/3}\log x).
\end{equation}
Ma and Sun \cite{MS,MScor} generalized Wu's work by showing that
\begin{equation}\label{1''}
S_f(x):=
\sum_{n\leq x}\frac{f([x/n])}{[x/n]}=x\sum_{m\geq1}\frac{f(m)}{m^2(m+1)}+  O(x^{1/3}\log x)
\end{equation}
for $f\in \{ \varphi, \Psi,   \sigma, \beta \}$,
where $\varphi$ denotes the Euler totient function, $\Psi$ donotes the Dedekind function, $\sigma$ donotes the sum-of-divisors function,
and $\beta$ donotes the alternating sum-of-divisors function.

In this paper, we will consider a more general case of \eqref{1''}, and give a uniform treatment.
Define id$(n)=n$ and $\mathbb{1}(n)=1$ for $n\geq1$.
Let $f$ be an arithmetic function and $g$ be the arithmetic function such that $f=\textrm{id}\ast g$.
Our main result is as follows.

\begin{theorem}\label{thm1}
	Let $f =\textup{id}\ast g$ be an arithmetic function with
\begin{align}\label{<1}
	&g(n)\ll 1
\end{align}
for $n\geq1$.
Then,
	\begin{equation}\label{jielun1}
		\sum_{n\leq x}\frac{f([x/n])}{[x/n]}=x\sum_{m\geq1}\frac{f(m)}{m^2(m+1)}+O(x^{1/3}\log x)
	\end{equation}
	as $x\rightarrow\infty$.
\end{theorem}

More generally, we have

\begin{theorem}\label{thm2}
Let $f =\textup{id}\ast g$ be an arithmetic function with 
\begin{align}\label{3''}
&g(n)\ll_\varepsilon n^\varepsilon
\end{align}
for $n\geq1$. 
Then,
\begin{equation}\label{jielun2}
\sum_{n\leq x}\frac{f([x/n])}{[x/n]}=
	 x\sum_{m\geq1}\frac{f(m)}{m^2(m+1)}+O(x^{1/3+\varepsilon})
\end{equation}
as $x\rightarrow\infty$ for any $\varepsilon >0$.
\end{theorem}

Notice that
$$
\varphi=\textrm{id} \ast\mu,\qquad
\sigma=\textrm{id} \ast \mathbb{1}, \qquad
\beta=\textrm{id} \ast (-1)^{\Omega(n)},\qquad
\Psi=  \textrm{id} \ast \mu^2,
$$
and each of
$\mu(n)$, $\mathbb{1}(n)$, $(-1)^{\Omega(n)}$ and $\mu^2(n)$ satisfies \eqref{3''}.
Therefore, Theorem \ref{thm1} implies the  main theorem in \cite{MS,MScor}.

Let $\mathbb{P}$ be the set of all primes and define
\begin{equation}\label{P}
\mathbb{1}_{\mathbb{P}}(n)=\begin{cases}
1, & \text{if } n\in\mathbb{P},\\
0, & \text{otherwise}.
\end{cases}
\end{equation}
Then $\mathbb{1}_{\mathbb{P}}(n) \ll 1$.
Let $f=\textup{id} \ast \mathbb{1}_{\mathbb{P}}$.
Then Theorem \ref{thm1} implies the following corollary.

\begin{corollary}
Let $f=\textup{id} \ast \mathbb{1}_{\mathbb{P}}$ and $ \mathbb{1}_{\mathbb{P}}$ is defined in \eqref{P}.
Then 
\[
\sum_{n\leqslant x}\frac{f([x/n])}{[x/n]}
= x\sum_{m\ge 1}\frac{f(m)}{m^{2}(m+1)} + O(x^{1/3}\log x)
\]
as $x\rightarrow \infty$.
\end{corollary}

\noindent\textbf{Notations.}
We write $e(t)$ for $e^{2\pi it}$, $m\sim M$ for  $M<m\leqslant 2M$ and $\psi(t) =\{t\}-\frac{1}{2}$, where $\{t\}$ denotes the fractional part of the real number $t$.

\vskip 5mm

\section{Some lemmas}

In this section, we will introduce  two lemmas which will be needed in our proof.
The first one is due to Valaler (\cite [Theorem A.6] {GrahamKolesnik1991}).

\begin{lemma}\label{lem:1}
For~$x\geqslant{}1$ and ~$H\geqslant{}1$, we have
$$\psi(x)=-\sum_{1\leqslant{} |h|\leqslant{} H} \Phi\left(\frac{h}{H+1}\right)\frac{e(hx)}{2\pi ih}+R_H(x),$$where~$\Phi(t):=\pi t(1-|t|)\cot(\pi t)+|t|$, and the error term ~$R_H(x)$ satisfies
\begin{equation}\label{RH}
|R_H(x)|\leqslant{}\frac{1}{2H+2}\sum_{|h|\leqslant{} H}\left(1-\frac{|h|}{H+1}\right)e(hx).
\end{equation}
\end{lemma}

The second one is  a direct corollary  of \cite [Proposition 3.1] {LWY-1}.

\begin{lemma}\label{lem:2}
Let $\alpha,\beta,\gamma>0$ and $\delta\in \mathbb{R}$ be some constants. For $X>0$, $H,M,N\geqslant 1$, define
\begin{equation}\label{S_delta}
S_{\delta}=S_{\delta}(H,M,N):=\sum_{h\sim H}\sum_{m\sim M}\sum_{n\sim N}a_{h,n}b_{m}e\left(X\frac{M^{\beta }N^{\gamma}}{H^{\alpha}}\frac{h^{\alpha}}{m^{\beta}n^{\gamma}+\delta}\right),
\end{equation}
where $a_{h,m}, b_{n} \in \mathbb{C}$ such that $|a_{h,m}|,|b_{n}|\le 1$.
For any $\epsilon>0$, we have
\begin{equation}\label{S_delta-z}
S_{\delta}\ll \big((X^{\kappa}H^{2+\kappa}M^{1+\kappa+\lambda}N^{2+\kappa})^{1/(2+2\kappa)}+HM^{1/2}N+H^{1/2}MN^{1/2}+X^{-1/2}HMN\big)X^{\varepsilon}
\end{equation}
uniformly for $H\leqslant N^{\gamma-1}M^{\beta}$  and $0\leqslant \delta\leqslant 1/\varepsilon$, where $(\kappa,\lambda)$ is an exponent pair and the implied constant depends on $(\alpha,\beta,\gamma,\varepsilon)$ only.
\end{lemma}

\vskip 5mm

\section{A key estimate}

Define
$$
\mathfrak{S}_{\delta}^{f}(x,D):=\sum_{D<d\leqslant 2D}\frac{f(d)}{d}\psi\left(\frac{x}{d+\delta}\right).
$$
From $f=\textrm{id}\ast g$, we  get
\begin{equation}\label{f(d)/d}
\frac{f(d)}{d}=\sum_{m|d}\frac{g(m)}{m},
\end{equation}
so that  we can decompose $\mathfrak{S}_{\delta}^{f}(x,D)$ into  bilinear forms
\begin{equation}\label{bliinear forms}
\mathfrak{S}_{\delta}^{f}(x,D)=\sum_{D<mn\leq 2D}\frac {g(m)}m\psi\left(\frac{x}{mn+\delta}\right).
\end{equation}
In this section we will get the following estimation of $\mathfrak{S}_{\delta}^{f}(x,D)$ which will  play a key role in the proof of the main theorem.

\begin{proposition}\label{proposition}
Let $f =\textup{id}\ast g$ be an arithmetic function, 
$\delta\ge 0$ be a fixed constant, 
$(\kappa, \lambda)$ be an exponent pair.
Then
\begin{equation}\label{key estimate}
	\mathfrak{S}_{\delta}^{f}(x,D)\ll
	(x^\kappa D^{-\kappa+\lambda})^{{1}/(1+\kappa)}      + x^\kappa D^{  -2\kappa+\lambda} \log x     + x^{-1}D^{2}
\end{equation}
uniformly for  $1\leqslant D\leqslant x$, if $g(n)\ll 1$; 
\begin{equation}\label{key estimate2}
	\mathfrak{S}_{\delta}^{f}(x,D)\ll
	D^\varepsilon\big(    (x^\kappa D^{-\kappa+\lambda})^{{1}/(1+\kappa)}      + x^\kappa D^{  -2\kappa+\lambda} \log x     + x^{-1}D^{2}    \big)
\end{equation}
uniformly for  $1\leqslant D\leqslant x$, if $g(n)\ll 1$. 
\end{proposition}

\begin{proof}
	From \eqref{f(d)/d}, then we can be written \eqref{bliinear forms} as
	$$
	\mathfrak{S}_{\delta}^{f}(x,D)=\sum_{m\le 2D}\frac{g(m)}{m}\sum_{D/m<n\le 2D/m}\psi\left(\frac{x}{mn+\delta}\right).
	$$
	Applying lemma \ref{lem:1} and noticing  that $0<\Phi(t)<1$ for $0<|t|<1$ and $g(m)\ll m^\varepsilon$ for $m\geq1$, 
we can derive
	\[
	\mathfrak{S}_{\delta}^{f}(x,D)
\ll\sum_{m\le  2D}\frac {g(m)}{m}
\bigg(\frac{(D/m)}{H}+\sum_{1\le h\le  H}\frac 1h \bigg|\sum_{D/m<n\le  2D/m}e \bigg(   \frac{hx}{mn+\delta}  \bigg)\bigg|\bigg)
	\]
	for $1\le  H\le  D/m$. 
Applying the exponent pair $(\kappa,\lambda)$ to the sum over $n$, 
we find that
	\begin{equation*}
		\begin{split}
			\mathfrak{S}_{\delta}^{f}(x,D)&\ll\sum_{m\le  2D}\frac {g(m)}{m}    \bigg(\frac{(D/m)}{H}+\sum_{1\le h\le  H}\frac 1h
\bigg\{ \bigg(   \frac{hx}{D^2/m}\bigg)^\kappa(D/m)^\lambda+\frac{(D^2/m)}{hx}  \bigg\}\bigg)\\
			&\ll\sum_{m\le  2D}\frac{g(m)}{m}  \bigg(  (D/m)H^{-1}+x^\kappa H^\kappa (D^2/m)^{-\kappa}(D/m)^\lambda+x^{-1}(D^2/m)   \bigg)
		\end{split}
	\end{equation*}
	for all $H\in[1,D/m]$. Optimising the parameter $H$ over $[1,D/m]$ we obtain
	\[
	\mathfrak{S}_{\delta}^{f}(x,D)\ll\sum_{m\le  2D}\frac {g(m)}{m}
	\big( (x^\kappa D^{\lambda-\kappa} m^{-\lambda})^{1/(1+\kappa)}     +x^\kappa (D^2/m)^{-\kappa}(D/m)^\lambda+x^{-1}(D^2/m)\big),
	\]
	which implies \eqref{key estimate} and \eqref{key estimate}.
\end{proof}

\vskip3mm
\section{Proof of Theorem  \ref{thm1}}

Let $f=\text{id} \ast g$ and $N_{f}\in [1,x^{\frac{1}{3}})$ be a  parameter to be chosen later.
Firstly, we write
\begin{equation}\label{1S3+S4}
\sum_{n\leqslant x}\frac{f([x/n])}{[x/n]}=S_{f}^{\dag}(x)+S_{f}^{\sharp}(x)
\end{equation}
with
\begin{equation*}
S_{f}^{\dag}(x):=\sum_{n\leqslant N_{f}}\frac{f([x/n])}{[x/n]},\qquad
S_{f}^{\sharp}(x):=\sum_{N_{f}<n \leqslant x}\frac{f([x/n])}{[x/n]}.
\end{equation*}

\vskip 1mm

Secondly, we  bound  $S_{f}^{\sharp}(x)$.
Put $m=[x/n]$. Then $x/(m+1)<n\leqslant x/m$.
Thus
\begin{equation}\label{S4fenjie1}
\begin{split}
S_{f}^{\sharp}(x)
&=\sum_{N_{f}<n\leqslant x}\frac{f([x/n])}{[x/n]}=\sum_{m\leqslant x/N_{f}}\frac{f(m)}{m}\sum_{x/(m+1)<n\leqslant x/m}1\\
&=\sum_{m\leqslant x/N_{f}}\frac{f(m)}{m}\left(\frac{x}{m}-\psi\left(\frac{x}{m}\right)-\frac{x}{m+1}+\psi\left(\frac{x}{m+1}\right)\right)\\
&=x\sum_{m\geqslant{} 1}\frac{f(m)}{m^{2}(m+1)}  -  {R}_{0}^{f}(x)  +  {R}_{1}^{f}(x)   + O(N_{f}),
\end{split}
\end{equation}
where  \begin{equation*}
	R_{\delta}^{f}(x):=\sum_{N_{f}<m\leq{x}/{N}_{f}}\frac{f(m)}{m}\psi \left(\frac{x}{m+\delta}\right)
\end{equation*}
for $\delta=0,1$, and we have used the following bounds
\begin{equation*}
\begin{split}
&x\sum_{m>{x}/{N_{f}}}\frac{f(m)}{m^{2}(m+1)}
\ll N_{f},
\qquad
\sum_{m\leqslant N_{f}}\frac{f(m)}{m}\psi \left(\frac{x}{m+\delta}\right)
\ll N_{f}.
\end{split}
\end{equation*}
Writing $D_j:=x/(2^jN_f)$, we have $N_f\le D_j\le x/N_f\le x$ for $0\le j\le \log(x/N_f^2)/\log 2$.
Thus we can apply \eqref{key estimate} of Proposition \ref{proposition} with  $(\kappa, \lambda) = (1/2, 1/2)$  to get
\begin{equation*}\label{R_delta-f2}
	\begin{split}
		|R_{\delta}^{f}(x)|
		&\le   \sum_{0\le   j \le  {\log(x/N_{f}^{2})}/{\log 2}}|\mathfrak{S}_{\delta}^{f}(x,D_{j})|
\\
&\ll   \sum_{0\le   j \le  {\log(x/N_{f}^{2})}/{\log 2}}
(  x^{1/3}       + x^{1/{2}}D_j^{-1/2} \log x     + x^{-1}D_j^{2}  )
\\
		&\ll                    (x^{1/3}     +  x^{1/{2}}N_f^{-1/2}     + x N_f^{-2}  )\log x.
	\end{split}
\end{equation*}
Inserting this into \eqref{S4fenjie1} and taking  $N_{f}=x^{ 1/3 }$, we derive
\begin{equation}\label{1S4}
	S_{f}^{\sharp}(x)=x\sum_{m\geqslant{} 1}\frac{f(m)}{m^{2}(m+1)}+O(x^{1/3} \log x ).
\end{equation}

Finally,  $ {f(d)}/{d}=\sum_{mn=d} {g(m)}/{m}\le \sum_{m|d} \frac {1}{m} \le \log d$
implies  $S_f^{\dag}(x)\ll N_f \log x $.
Inserting this and \eqref{1S4} into \eqref{1S3+S4}, we  get \eqref{jielun1}.

\vskip3mm
\section{Proof of Theorem  \ref{thm2}}

Let $f=\text{id} \ast g$ and $N_{f}\in [1,x^{\frac{1}{3}})$ be a  parameter to be chosen later.
Firstly, we write
\begin{equation}\label{S3+S4}
\sum_{n\leqslant x}\frac{f([x/n])}{[x/n]}=S_{f}^{\dag}(x)+S_{f}^{\sharp}(x)
\end{equation}
with
\begin{equation*}
S_{f}^{\dag}(x):=\sum_{n\leqslant N_{f}}\frac{f([x/n])}{[x/n]},\qquad
S_{f}^{\sharp}(x):=\sum_{N_{f}<n \leqslant x}\frac{f([x/n])}{[x/n]}.
\end{equation*}

\vskip 1mm

Secondly, we  bound  $S_{f}^{\sharp}(x)$.
Put $m=[x/n]$. Then $x/(m+1)<n\leqslant x/m$.
Thus
\begin{equation}\label{S4fenjie}
\begin{split}
S_{f}^{\sharp}(x)
&=\sum_{N_{f}<n\leqslant x}\frac{f([x/n])}{[x/n]}=\sum_{m\leqslant x/N_{f}}\frac{f(m)}{m}\sum_{x/(m+1)<n\leqslant x/m}1\\
&=x\sum_{m\geqslant{} 1}\frac{f(m)}{m^{2}(m+1)}+ O(N_{f}x^{\varepsilon})-{R}_{0}^{f}(x)+{R}_{1}^{f}(x),
\end{split}
\end{equation}
where  \begin{equation*}
	R_{\delta}^{f}(x):=\sum_{N_{f}<m\leq{x}/{N}_{f}}\frac{f(m)}{m}\psi \left(\frac{x}{m+\delta}\right)
\end{equation*}
for $\delta=0,1$, and we have used the following bounds
\begin{equation*}
\begin{split}
&x\sum_{m>{x}/{N_{f}}}\frac{f(m)}{m^{2}(m+1)}\ll_{\varepsilon} x\sum_{m>{x}/{N_{f}}}\frac{m^{1+\varepsilon}}{m^{2}(m+1)}\ll N_{f}x^{\varepsilon},\\
&\sum_{m\leqslant N_{f}}\frac{f(m)}{m}\psi \left(\frac{x}{m+\delta}\right)\ll_{\varepsilon} \sum_{m\leqslant N_f}\frac{f(m)}{m}\ll N_{f}x^{\varepsilon}.
\end{split}
\end{equation*}
Writing $D_j:=x/(2^jN_f)$, we have $N_f\le D_j\le x/N_f\le x$ for $0\le j\le \log(x/N_f^2)/\log 2$.
Thus we can apply \eqref{key estimate2} of Proposition \ref{proposition} with  $(\kappa, \lambda) = (1/2, 1/2)$  to get
\begin{equation*}\label{R_delta-f}
	\begin{split}
		|R_{\delta}^{f}(x)|
		&\le   \sum_{0\le   j \le  {\log(x/N_{f}^{2})}/{\log 2}}|\mathfrak{S}_{\delta}^{f}(x,D_{j})|
\\
&\ll   \sum_{0\le   j \le  {\log(x/N_{f}^{2})}/{\log 2}}
(  x^{1/3}       + x^{1/{2}}D_j^{-1/2} \log x     + x^{-1}D_j^{2}  )D_j^\varepsilon
\\
		&\ll                    (x^{1/3}     +  x^{1/{2}}N_f^{-1/2}     + x N_f^{-2}  )    x^\varepsilon.
	\end{split}
\end{equation*}
Inserting this into \eqref{S4fenjie} and taking  $N_{f}=x^{ 1/3 }$, we derive
\begin{equation}\label{S4}
	S_{f}^{\sharp}(x)=x\sum_{m\geqslant{} 1}\frac{f(m)}{m^{2}(m+1)}+O(x^{1/3+\varepsilon} ).
\end{equation}

Finally,  $ {f(d)}/{d}=\sum_{mn=d} {g(m)}/{m}\le \sum_{m|d} \frac {m^\varepsilon}{m} \le d^{\varepsilon}$
implies  $S_f^{\dag}(x)\ll N_f x^{\varepsilon}$.
Inserting this and \eqref{S4} into \eqref{S3+S4}, we  get \eqref{jielun2}.

\vskip 5mm
\noindent\textbf{Acknowledgments}\
This work is supported  by the National Natural Science Foundation of China (Grant No. 11771252).

\end{document}